\documentclass[12pt]{article}
\usepackage{amssymb}
\usepackage{amsthm}
\usepackage{amsmath}
\usepackage{graphicx}
\usepackage{enumerate}

\newtheorem{theorem}{Theorem}[section]

\newtheorem{lemma}[theorem]{Lemma}
\newtheorem{proposition}[theorem]{Proposition}
\newtheorem{remark}[theorem]{Remark}
\newtheorem{example}[theorem]{Example}

\title{Conditions forcing the existence of relative complements in lattices and posets}
\author{Ivan~Chajda and Helmut~L\"anger}
\date{}
\begin{document}

\footnotetext{Support of the research by the Austrian Science Fund (FWF), project I~4579-N, and the Czech Science Foundation (GA\v CR), project 20-09869L, entitled ``The many facets of orthomodularity'', as well as by \"OAD, project CZ~02/2019, entitled ``Function algebras and ordered structures related to logic and data fusion'', and, concerning the first author, by IGA, project P\v rF~2021~030, is gratefully acknowledged.}

\maketitle

\begin{abstract}
It is elementary and well-known that if an element $x$ of a bounded modular lattice $\mathbf L$ has a complement in $\mathbf L$ then $x$ has a relative complement in every interval $[a,b]$ containing $x$. We show that the relatively strong assumption of modularity of $\mathbf L$ can be replaced by a weaker one formulated in the language of so-called modular triples. We further show that, in general, we need not suppose that $x$ has a complement in $\mathbf L$. By introducing the concept of modular triples in posets, we extend our results obtained for lattices to posets. It should be remarked that the notion of a complement can be introduced also in posets that are not bounded.
\end{abstract}



\section{Introduction}

Let $\mathbf L=(L,\vee,\wedge,0,1)$ be a bounded lattice and $x\in L$. Then $y\in L$ is called a {\em complement} of $x$ if $x\vee y=1$ and $x\wedge y=0$. The lattice $\mathbf L$ is called {\em complemented} if every element of $L$ has a complement. Of course, $x$ need not have a unique complement. However, if $\mathbf L$ is distributive then every element of $L$ has at most one complement (see e.g.\ \cite B). We can prove a similar result under a bit weaker condition in the following lemma. For this, let us introduce the following concept.

Let $(L,\vee,\wedge)$ be a lattice and $a,b,c\in L$. We call $(a,b,c)$ a {\em distributive} triple if $(a\vee b)\wedge c=(a\wedge c)\vee(b\wedge c)$.

\begin{lemma}\label{lem1}
Let $\mathbf L=(L,\vee,\wedge,0,1)$ be a bounded lattice and $x\in L$. Assume that $y$ and $z$ are complements of $x$ and that $(x,y,z)$ and $(x,z,y)$ are distributive triples. Then $y=z$.
\end{lemma}

\begin{proof}
We have
\begin{align*}
y & =1\wedge y=(x\vee z)\wedge y=(x\wedge y)\vee(z\wedge y)=0\vee(z\wedge y)=z\wedge y=y\wedge z=0\vee(y\wedge z)= \\
  & =(x\wedge z)\vee(y\wedge z)=(x\vee y)\wedge z=1\wedge z=z.
\end{align*}
\end{proof}

Let $\mathbf L=(L,\vee,\wedge)$ be a lattice, $a,b\in L$ with $a\leq b$ and $x\in[a,b]$. An element $z$ of $[a,b]$ is called a {\em relative complement} of $x$ in $[a,b]$ if $x\vee z=b$ and $x\wedge z=a$, i.e.\ $z$ is a complement of $x$ in the sublattice $([a,b],\vee,\wedge)$ of $\mathbf L$. Let $R(a,b,x)$ denote the set of all relative complements of $x$ in $[a,b]$. The lattice $\mathbf L$ is called {\em relatively complemented} if for any $a,b\in L$ with $a\leq b$ and each $x\in[a,b]$ we have $R(a,b,x)\neq\emptyset$.

It is well-known that if $\mathbf L=(L,\vee,\wedge,0,1)$ is a bounded modular lattice, $a,b\in L$ with $a\leq b$, $x\in[a,b]$ and $y$ is a complement of $x$ then $e:=(a\vee y)\wedge b=a\vee(y\wedge b)\in R(a,b,x)$. Hence every complemented modular lattice is relatively complemented. However, modularity is only a sufficient condition but not necessary. There exist non-modular complemented lattices where every interval is also complemented, i.e.\ the lattice is relatively complemented. An example of such a non-modular lattice is depicted in Fig.~1.

\vspace*{-2mm}

\begin{center}
\setlength{\unitlength}{7mm}
\begin{picture}(10,8)
\put(5,1){\circle*{.3}}
\put(5,3){\circle*{.3}}
\put(7,3){\circle*{.3}}
\put(1,5){\circle*{.3}}
\put(3,5){\circle*{.3}}
\put(7,5){\circle*{.3}}
\put(9,5){\circle*{.3}}
\put(5,7){\circle*{.3}}
\put(5,1){\line(-1,1)4}
\put(5,1){\line(0,1)2}
\put(5,1){\line(1,1)4}
\put(5,7){\line(-2,-1)4}
\put(5,7){\line(-1,-1)2}
\put(5,7){\line(1,-1)2}
\put(5,7){\line(2,-1)4}
\put(5,3){\line(-1,1)2}
\put(5,3){\line(1,1)2}
\put(7,3){\line(0,1)2}
\put(4.85,.25){$0$}
\put(5.4,2.85){$d$}
\put(7.4,2.85){$a$}
\put(.3,4.85){$b$}
\put(2.3,4.85){$f$}
\put(7.4,4.85){$c$}
\put(9.4,4.85){$e$}
\put(4.85,7.4){$1$}
\put(4.2,-.75){{\rm Fig.~1}}
\end{picture}
\end{center}

\vspace*{4mm}

The aim of this paper is not to characterize relatively complemented lattices and posets but to provide conditions under which an element $x$ of $[a,b]$ has a relative complement in this interval. Of course, $x$ can have more than one relative complement, but not all of them can be obtained by our construction.

As mentioned above, if $\mathbf L$ is a distributive complemented lattice then every element of $L$ has just one complement and $\mathbf L$ is called {\em Boolean}. The converse holds only in a finite lattice, an example of a non-distributive lattice with unique complementation was constructed by R.~P.~Dilworth (\cite D). Lattices with unique complementation were also studied by the first author and R.~Padmanabhan (\cite{CP}).

Let $(L,\vee,\wedge)$ be a lattice and $a,b,c\in L$. Recall that the triple $(a,b,c)$ is called {\em modular} if $a\leq c$ and $(a\vee b)\wedge c=a\vee(b\wedge c)$.

The concept of complementation was transferred to posets by the first author in \cite C.

Let $(P,\leq)$ be a poset $a,b\in P$ and $A,B\subseteq P$. We say $A<B$ if $x\leq y$ for all $x\in A$ and $y\in B$. Instead of $\{a\}<\{b\}$, $\{a\}<B$ and $A<\{b\}$ we simply write $a<b$, $a<B$ and $A<b$, respectively. Analogously we proceed with the relational symbols $\leq$, $>$ and $\geq$. Denote by
\begin{align*}
L(A) & :=\{x\in P\mid x\leq A\}\text{ and} \\
U(A) & :=\{x\in P\mid A\leq x\}
\end{align*}
the so-called {\em lower} and {\em upper cone} of $A$, respectively. Instead of $L(\{a\})$, $L(\{a,b\})$, $L(A\cup\{a\})$, $L(A\cup B)$ and $L\big(U(A)\big)$ we simply write $L(a)$, $L(a,b)$, $L(A,a)$, $L(A,B)$ and $LU(A)$, respectively. Analogously, we proceed in similar cases. The element $b$ is called a {\em complement} of $a$ if $LU(a,b)=UL(a,b)=P$ (see e.g.\ \cite C). Of course, if $(P,\leq,0,1)$ is a bounded poset then $b$ is a complement of $a$ if and only if $U(a,b)=\{1\}$ and $L(a,b)=\{0\}$, i.e.\ $a\vee b=1$ and $a\wedge b=0$.

The concept of a modular poset was introduced by Larmerov\'a and Rach\r unek (\cite{LR}) as follows:

A {\em poset} $\mathbf P=(P,\leq)$ is called {\em modular} if
\[
L\big(U(x,y),z\big)=LU\big(x,L(y,z)\big)
\]
for all $x,y,z\in P$ with $x\leq z$. Complements in posets were investigated in \cite C and \cite{CLP}. If $(P,\leq)$ is a poset, $a,b\in P$ with $a\leq b$ and $x\in[a,b]$ then $y\in[a,b]$ is called a {\em relative complement of $x$ in $[a,b]$} if
\[
U(x,y)=U(b)\text{ and }L(x,y)=L(a).
\]
Of course, this is equivalent to $x\vee y=b$ and $x\wedge y=a$. Again $R(a,b,x)$ denotes the set of all relative complements of $x$ in $[a,b]$. Relative complements in posets were already treated by the first author and Mor\'avkov\'a in \cite{CM}. Recall that a subset $A$ of $P$ is called {\em convex} if $a,b,c\in P$, $a\leq b\leq c$ and $a,c\in A$ together imply $b\in A$. Clearly, the set $R(a,b,x)$ of all relative complements of $x$ in $[a,b]$ is convex.

\section{Relative complements in lattices}

As promised in the introduction, we show that there are conditions weaker than modularity ensuring that an element $x$ of $[a,b]$ has a relative complement. We are going to state two such results. The following result is a special case of Proposition~\ref{prop1} which was proved in \cite{CM}. For the sake of completeness we provide a proof.

\begin{theorem}\label{th1}
Let $\mathbf L=(L,\vee,\wedge)$ be a lattice, $a,b\in L$ with $a\leq b$, $x\in[a,b]$, $y\in L$, $e:=(a\vee y)\wedge b$ and $f:=a\vee(y\wedge b)$. Then $f\leq e$ and the following are equivalent:
\begin{enumerate}[{\rm(i)}]
\item $e,f\in R(a,b,x)$,
\item $(a\vee y)\wedge x=a$ and $x\vee(y\wedge b)=b$.
\end{enumerate}
\end{theorem}

\begin{proof}
Because of $a\leq a\vee y$, $a\leq b$, $y\wedge b\leq a\vee y$ and $y\wedge b\leq b$ we have $f\leq e$. Moreover, we have
\begin{align*}
        a & \leq f\leq e\leq b, \\
e\wedge x & =\big((a\vee y)\wedge b\big)\wedge x=(a\vee y)\wedge x, \\
  x\vee f & =x\vee\big(a\vee(y\wedge b)\big)=x\vee(y\wedge b).
\end{align*}
Since $a\leq x$ and $a\leq f\leq e$, we have that $x\wedge e=a$ implies $x\wedge f=a$, and since $x\leq b$ and $f\leq e\leq b$, we have that $x\vee f=b$ implies $x\vee e=b$. Hence (i) and (ii) are both equivalent to $e\wedge x=a$ and $x\vee f=b$.
\end{proof}

It is worth noticing that we do not assume $y$ to be a complement of $x$.

\begin{example}
Consider the lattice $\mathbf L$ depicted in Fig.~2:

\vspace*{-2mm}

\begin{center}
\setlength{\unitlength}{7mm}
\begin{picture}(12,14)
\put(6,1){\circle*{.3}}
\put(4,3){\circle*{.3}}
\put(8,3){\circle*{.3}}
\put(6,5){\circle*{.3}}
\put(1,6){\circle*{.3}}
\put(11,6){\circle*{.3}}
\put(6,7){\circle*{.3}}
\put(4,9){\circle*{.3}}
\put(8,9){\circle*{.3}}
\put(6,11){\circle*{.3}}
\put(6,13){\circle*{.3}}
\put(6,1){\line(-1,1)5}
\put(6,1){\line(1,1)5}
\put(6,5){\line(-1,-1)2}
\put(6,5){\line(0,1)2}
\put(6,5){\line(1,-1)2}
\put(6,7){\line(-1,1)2}
\put(6,7){\line(1,1)2}
\put(6,11){\line(-1,-1)5}
\put(6,11){\line(0,1)2}
\put(6,11){\line(1,-1)5}
\put(5.85,.25){$0$}
\put(3.3,2.85){$a$}
\put(8.4,2.85){$c$}
\put(6.4,4.85){$f$}
\put(.3,5.85){$x$}
\put(11.4,5.85){$y$}
\put(6.4,6.85){$e$}
\put(3.3,8.85){$b$}
\put(8.4,8.85){$d$}
\put(6.4,10.85){$g$}
\put(5.85,13.4){$1$}
\put(5.2,-.75){{\rm Fig.~2}}
\end{picture}
\end{center}

\vspace*{4mm}

Evidently, $\mathbf L$ is neither modular nor complemented. Further, $x\in[a,b]$, and for $y\in L$ we have
\[
(a\vee y)\wedge x=a\text{ and }x\vee(y\wedge b)=b.
\]
Hence, the assumptions of Theorem~\ref{th1} are satisfied. Put $e:=(a\vee y)\wedge b$ and $f:=a\vee(y\wedge b)$. Then $e,f\in R(a,b,x)$. It is worth noticing that neither $y$ belongs to $[a,b]$ nor it is a complement of $x$.
\end{example}

If $y$ is, moreover, a complement of $x$ then we can derive relative complements of $x$ in $[a,b]$ by using modular triples.

\begin{theorem}\label{th4}
Let $\mathbf L=(L,\vee,\wedge,0,1)$ be a bounded lattice, $a,b\in L$ with $a\leq b$, $x\in[a,b]$, $y$ a complement of $x$, $e:=(a\vee y)\wedge b$ and $f:=a\vee(y\wedge b)$. Then $f\leq e$ and the following are equivalent:
\begin{enumerate}[{\rm(i)}]
\item $e,f\in R(a,b,x)$,
\item $(a,y,x)$ and $(x,y,b)$ are modular triples.
\end{enumerate}
\end{theorem}

\begin{proof}
Since $a\leq a\vee y$, $a\leq b$, $y\wedge b\leq a\vee y$ and $y\wedge b\leq b$, we have $f\leq e$. Because of Theorem~\ref{th1}, (i) is equivalent to
\begin{equation}\label{equ1}
(a\vee y)\wedge x=a\text{ and }b=x\vee(y\wedge b).
\end{equation}
Since
\begin{align*}
a & =a\vee0=a\vee(y\wedge x), \\
b & =1\wedge b=(x\vee y)\wedge b,
\end{align*}
(\ref{equ1}) is equivalent to
\begin{equation}\label{equ2}
(a\vee y)\wedge x=a\vee(y\wedge x)\text{ and }(x\vee y)\wedge b=x\vee(y\wedge b).
\end{equation}
Finally, due to the definition of modular triples, (\ref{equ2}) is equivalent to (ii).
\end{proof}

An example of the situation described by Theorem~\ref{th4} is the lattice depicted in Fig.~3 where $y$ is a complement of $x$ and the elements $e$ and $f$ satisfy the assumptions of Theorem~\ref{th4}. Let us note that $f<e$ in this case.

\begin{example}
Consider the lattice $\mathbf L$ depicted in Fig.~3:

\vspace*{-2mm}

\begin{center}
\setlength{\unitlength}{7mm}
\begin{picture}(12,12)
\put(6,1){\circle*{.3}}
\put(4,3){\circle*{.3}}
\put(8,3){\circle*{.3}}
\put(6,5){\circle*{.3}}
\put(1,6){\circle*{.3}}
\put(11,6){\circle*{.3}}
\put(6,7){\circle*{.3}}
\put(4,9){\circle*{.3}}
\put(8,9){\circle*{.3}}
\put(6,11){\circle*{.3}}
\put(6,1){\line(-1,1)5}
\put(6,1){\line(1,1)5}
\put(6,5){\line(-1,-1)2}
\put(6,5){\line(0,1)2}
\put(6,5){\line(1,-1)2}
\put(6,7){\line(-1,1)2}
\put(6,7){\line(1,1)2}
\put(6,11){\line(-1,-1)5}
\put(6,11){\line(1,-1)5}
\put(5.85,.25){$0$}
\put(3.3,2.85){$a$}
\put(8.4,2.85){$c$}
\put(6.4,4.85){$f$}
\put(.3,5.85){$x$}
\put(11.4,5.85){$y$}
\put(6.4,6.85){$e$}
\put(3.3,8.85){$b$}
\put(8.4,8.85){$d$}
\put(5.85,11.4){$1$}
\put(5.2,-.75){{\rm Fig.~3}}
\end{picture}
\end{center}

\vspace*{4mm}

The lattice $\mathbf L$ is not modular, $y$ is a complement of $x$, the triples $(a,y,x)$ and $(x,y,b)$ are modular since
\begin{align*}
(a\vee y)\wedge x & =d\wedge x=a\text{ and }a\vee(y\wedge x)=a\vee0=a, \\
x\vee(y\wedge b) & =x\vee c=b\text{ and }(x\vee y)\wedge b=1\wedge b=b
\end{align*}
and $e=(a\vee y)\wedge b$ and $f=a\vee(y\wedge b)$ are relative complements of $x$ in $[a,b]$ in accordance with Theorem~\ref{th4}.
\end{example}

\begin{example}
Unfortunately, the methods presented in Theorems~\ref{th1} and \ref{th4} do not produce all relative complements of $x$ in $[a,b]$, even if the lattice $\mathbf L$ is modular. Consider the lattice $\mathbf L$ visualized in Fig.~4:

\vspace*{-2mm}

\begin{center}
\setlength{\unitlength}{7mm}
\begin{picture}(10,10)
\put(5,1){\circle*{.3}}
\put(3,3){\circle*{.3}}
\put(3,5){\circle*{.3}}
\put(7,3){\circle*{.3}}
\put(1,5){\circle*{.3}}
\put(5,5){\circle*{.3}}
\put(9,5){\circle*{.3}}
\put(3,7){\circle*{.3}}
\put(7,7){\circle*{.3}}
\put(5,9){\circle*{.3}}
\put(5,1){\line(-1,1)4}
\put(5,1){\line(1,1)4}
\put(5,9){\line(-1,-1)4}
\put(5,9){\line(1,-1)4}
\put(3,3){\line(0,1)4}
\put(3,3){\line(1,1)4}
\put(7,3){\line(-1,1)4}
\put(4.85,.25){$0$}
\put(2.3,2.85){$a$}
\put(7.4,2.85){$c$}
\put(.3,4.85){$x$}
\put(2.2,4.85){$z_1$}
\put(5.4,4.85){$z_2$}
\put(9.4,4.85){$y$}
\put(2.3,6.85){$b$}
\put(7.4,6.85){$d$}
\put(4.85,9.4){$1$}
\put(4.2,-.75){{\rm Fig.~4}}
\end{picture}
\end{center}

\vspace*{4mm}

Then $(a\vee y)\wedge x=d\wedge x=a$ and $x\vee(y\wedge b)=x\vee c=b$. Thus the assumptions of Theorem~\ref{th1} are satisfied. Hence, $(a\vee y)\wedge b=d\wedge b=z_2$ and $a\vee(y\wedge b)=a\vee c=z_2$ belong to $R(a,b,x)$. However, $x$ has in $[a,b]$ also another relative complement $z_1$ which is not obtained in this way. The reason is that $z_1$ is both join- and meet-irreducible and hence it cannot be a result of any term function composed by means of join and meet.
\end{example}

\begin{example}
On the other hand, $R(a,b,x)$ is a convex set, thus if we can compute elements $e$ and $f$ as shown by Theorem~\ref{th1} then also every element $z$ in the interval $[f,e]$ is a relative complement of $x$, see the lattice depicted in Fig.~5:

\vspace*{-2mm}

\begin{center}
\setlength{\unitlength}{7mm}
\begin{picture}(14,16)
\put(7,1){\circle*{.3}}
\put(5,3){\circle*{.3}}
\put(9,3){\circle*{.3}}
\put(7,5){\circle*{.3}}
\put(1,7){\circle*{.3}}
\put(7,7){\circle*{.3}}
\put(13,7){\circle*{.3}}
\put(7,9){\circle*{.3}}
\put(5,11){\circle*{.3}}
\put(9,11){\circle*{.3}}
\put(7,13){\circle*{.3}}
\put(7,15){\circle*{.3}}
\put(7,1){\line(-1,1)6}
\put(7,1){\line(1,1)6}
\put(7,13){\line(-1,-1)6}
\put(7,13){\line(0,1)2}
\put(7,13){\line(1,-1)6}
\put(7,5){\line(-1,-1)2}
\put(7,5){\line(0,1)4}
\put(7,5){\line(1,-1)2}
\put(7,9){\line(-1,1)2}
\put(7,9){\line(1,1)2}
\put(6.85,.25){$0$}
\put(4.3,2.85){$a$}
\put(9.4,2.85){$c$}
\put(7.4,4.85){$f$}
\put(.3,6.85){$x$}
\put(7.4,6.85){$z$}
\put(13.4,6.85){$y$}
\put(7.4,8.85){$e$}
\put(4.3,10.85){$b$}
\put(9.4,10.85){$d$}
\put(7.4,12.85){$e$}
\put(6.85,15.4){$1$}
\put(6.2,-.75){{\rm Fig.~5}}
\end{picture}
\end{center}

\vspace*{4mm}

The assumptions of Theorem~\ref{th1} are satisfied, i.e.\ $(a\vee y)\wedge x=d\wedge x=a$ and $x\vee(y\wedge b)=x\vee c=b$. Thus $e,f\in R(a,b,x)$. Here $f<z<e$, thus $z\in R(a,b,x)$.
\end{example}

If $y$ is a complement of $x$ and the triple $(a,y,b)$ is modular, we can determine a relative complement of $x$ in $[a,b]$ by using further modular triples. In this case we have $e=f$, i.e.\ we compute only one relative complement of $x$ in $[a,b]$.

\begin{theorem}\label{th3}
Let $\mathbf L=(L,\vee,\wedge,0,1)$ be a bounded lattice, $a,b\in L$ with $a\leq b$, $x\in[a,b]$ and $y$ a complement of $x$, assume $(a,y,b)$ to be a modular triple and put $e:=(a\vee y)\wedge b=a\vee(y\wedge b)$. Then the following are equivalent:
\begin{enumerate}[{\rm(i)}]
\item $e\in R(a,b,x)$,
\item $(a,y\wedge b,x)$ and $(x,a\vee y,b)$ are modular triples.
\end{enumerate}
\end{theorem}

\begin{proof}
Because of Theorem~\ref{th1}, (i) is equivalent to
\begin{equation}\label{equ3}
\big(a\vee(y\wedge b)\big)\wedge x=a\text{ and }b=x\vee\big((a\vee y)\wedge b\big).
\end{equation}
Since
\begin{align*}
a & =a\vee0=a\vee(y\wedge x)=a\vee\big((y\wedge b)\wedge x\big), \\
b & =1\wedge b=(x\vee y)\wedge b=\big(x\vee(a\vee y)\big)\wedge b,
\end{align*}
(\ref{equ3}) is equivalent to
\begin{equation}\label{equ4}
\big(a\vee(y\wedge b)\big)\wedge x=a\vee\big((y\wedge b)\wedge x\big)\text{ and }\big(x\vee(a\vee y)\big)\wedge b=x\vee\big((a\vee y)\wedge b\big).
\end{equation}
Finally, due to the definition of modular triples, (\ref{equ4}) is equivalent to (ii).
\end{proof}

\begin{remark}
Consider the following well-known result mentioned in the introduction: If $\mathbf L=(L,\vee,\wedge,0,1)$ is a bounded modular lattice, $a,b\in L$ with $a\leq b$, $x\in[a,b]$ and $y$ a complement of $x$ then $e:=(a\vee y)\wedge b=a\vee(y\wedge b)\in R(a,b,x)$. In the proof of this result only modularity of the triples $(a,y,b)$, $(a,y,x)$ and $(x,y,b)$ is used. According to Theorem~\ref{th3}, modularity of these three triples must imply modularity of the triples $(a,y\wedge b,x)$ and $(x,a\vee y,b)$ which can be easily shown:
\begin{align*}
\big(a\vee(y\wedge b)\big)\wedge x & =\big((a\vee y)\wedge b\big)\wedge x=(a\vee y)\wedge x=a\vee(y\wedge x)=a\vee\big((y\wedge b)\wedge x\big), \\
  \big(x\vee(a\vee y)\big)\wedge b & =(x\vee y)\wedge b=x\vee(y\wedge b)=x\vee\big(a\vee(y\wedge b)\big)=x\vee\big((a\vee y)\wedge b\big).
\end{align*}
\end{remark}

\section{Relative complements in posets}

Now we turn our attention to complements and relative complements in posets. For this purpose, we must define again distributive and modular triples.

Let $\mathbf P=(P,\leq)$ be a poset and $a,b,c\in P$. Then the poset $\mathbf P$ is called {\em distributive} if
\[
L\big(U(x,y),z\big)=LU\big(L(x,z),L(y,z)\big)
\]
for all $x,y,z\in P$. We call $(a,b,c)$ a {\em distributive} triple of $\mathbf P$ if
\[
L\big(U(a,b),c\big)=LU\big(L(a,c),L(b,c)\big)
\]
and we call $(a,b,c)$ a {\em modular triple} of $\mathbf P$ if $a\leq c$ and
\[
L\big(U(a,b),c\big)=LU\big(a,L(b,c)\big).
\]
The following result is a straightforward generalization of the corresponding result for lattices.

\begin{proposition}
{\rm(\cite C)} Let $\mathbf P=(P,\leq)$ be a distributive poset and $x\in P$. Then $x$ has at most one complement.
\end{proposition}

Analogously as in Lemma~\ref{lem1}, distributivity of $(P,\leq)$ can be replaced by a weaker condition. The following lemma extends Lemma~\ref{lem1} to posets.

\begin{lemma}
Let $\mathbf P=(P,\leq)$ be a poset and $x\in L$. Assume that $y$ and $z$ are complements of $x$ and that $(x,y,z)$ and $(x,z,y)$ are distributive triples. Then $y=z$.
\end{lemma}

\begin{proof}
We have
\begin{align*}
L(y) & =P\cap L(y)=LU(x,z)\cap L(y)=L\big(U(x,z),y\big)=LU\big(L(x,y),L(z,y)\big)= \\
     & =L\big(UL(x,y)\cap UL(z,y)\big)=L\big(P\cap UL(z,y)\big)=LUL(z,y)=L(z,y)=L(y,z)= \\
     & =LUL(y,z)=L(P\cap UL(y,z)\big)=L\big(UL(x,z)\cap UL(y,z)\big)= \\
     & =LU\big(L(x,z),L(y,z)\big)=L\big(U(x,y),z\big)=LU(x,y)\cap L(z)=P\cap L(z)=L(z)
\end{align*}
and hence $y=z$.
\end{proof}
 
Relations between complements and relative complements in posets were investigated in \cite{CM}.

Again, modularity of $(P,\leq)$ can be replaced by some weaker conditions, see the following result.

\begin{proposition}\label{prop1}
{\rm(\cite{CM})} Let $\mathbf P=(P,\leq)$ be a poset, $a,b\in P$ with $a\leq b$, $x\in[a,b]$ and $y\in P$ and assume that $e$ is the greatest element of $L\big(U(a,y),b\big)$ and $f$ the smallest element of $U\big(a,L(y,b)\big)$. If
\begin{align*}
L\big(U(a,y),x\big) & =L(a), \\
U\big(x,L(y,b)\big) & =U(b)
\end{align*}
then $e,f\in R(a,b,x)$.
\end{proposition}

Let us note that $e$ is the greatest element of $L\big(U(a,y),b\big)$ if and only if $L\big(U(a,y),b\big)=L(e)$, and $f$ is the smallest element of $U\big(a,L(y,b)\big)$ if and only if $U\big(a,L(y,b)\big)=U(f)$.

Now we show that the converse of Proposition~\ref{prop1} is also true, i.e.\ we can state the following theorem which is an extension of Theorem~\ref{th1} to posets.

\begin{theorem}\label{th5}
Let $\mathbf P=(P,\leq)$ be a poset, $a,b\in P$ with $a\leq b$, $x\in[a,b]$ and $y\in P$ and assume that $e$ is the greatest element of $L\big(U(a,y),b\big)$ and $f$ the smallest element of $U\big(a,L(y,b)\big)$. Then $a\leq f\leq e\leq b$ and the following are equivalent:
\begin{enumerate}[{\rm(i)}]
\item $e,f\in R(a,b,x)$,
\item $L\big(U(a,y),x\big)=L(a)$ and $U\big(x,L(y,b)\big)=U(b)$.
\end{enumerate}
\end{theorem}

\begin{proof}
We have $L\big(U(a,y),b\big)=L(e)$ and $U\big(a,L(y,b)\big)=U(f)$. Because of $U(a,y)\cup\{b\}\subseteq U\big(a,L(y,b)\big)=U(f)$ we conclude
\[
f\in L(f)=LU(f)\subseteq L\big(U(a,y),b\big)=L(e)
\]
and hence $a\leq f\leq e\leq b$. Moreover, we have
\begin{align*}
L(e,x) & =L(e)\cap L(x)=L\big(U(a,y),b\big)\cap L(x)=LU(a,y)\cap L(b)\cap L(x)= \\
       & =LU(a,y)\cap L(x)=L\big(U(a,y),x\big), \\
U(x,f) & =U(x)\cap U(f)=U(x)\cap U\big(a,L(y,b)\big)=U(x)\cap U(a)\cap UL(y,b)= \\
       & =U(x)\cap UL(y,b)=U\big(x,L(y,b)\big).
\end{align*}
Since $a\leq x$ and $a\leq f\leq e$, we have that $L(x,e)=L(a)$ implies $L(c,f)=L(a)$, and since $x\leq b$ and $f\leq e\leq b$, we have that $U(x,f)=U(b)$ implies $U(x,e)=U(b)$. Hence (i) and (ii) are both equivalent to $L(e,x)=L(a)$ and $U(x,f)=U(b)$.
\end{proof}

\begin{example}
Consider the poset $\mathbf P$ depicted in Fig.~6:

\vspace*{-2mm}

\begin{center}
\setlength{\unitlength}{7mm}
\begin{picture}(12,12)
\put(6,1){\circle*{.3}}
\put(4,3){\circle*{.3}}
\put(6,3){\circle*{.3}}
\put(8,3){\circle*{.3}}
\put(6,5){\circle*{.3}}
\put(1,6){\circle*{.3}}
\put(11,6){\circle*{.3}}
\put(4,9){\circle*{.3}}
\put(6,9){\circle*{.3}}
\put(8,9){\circle*{.3}}
\put(6,11){\circle*{.3}}
\put(6,7){\circle*{.3}}
\put(6,1){\line(-1,1)5}
\put(6,1){\line(0,1){10}}
\put(6,1){\line(1,1)5}
\put(4,3){\line(1,1)2}
\put(8,3){\line(-1,1)2}
\put(6,5){\line(-1,-1)2}
\put(6,5){\line(1,-1)2}
\put(11,6){\line(-5,-3)5}
\put(11,6){\line(-5,3)5}
\put(6,7){\line(-1,1)2}
\put(6,7){\line(1,1)2}
\put(6,11){\line(-1,-1)5}
\put(6,11){\line(1,-1)5}
\put(5.85,.25){$0$}
\put(3.3,2.85){$a$}
\put(5.3,2.85){$c$}
\put(8.4,2.85){$g$}
\put(5.3,4.85){$f$}
\put(.3,5.85){$x$}
\put(5.3,6.85){$e$}
\put(3.3,8.85){$b$}
\put(5.3,8.85){$d$}
\put(11.4,5.85){$y$}
\put(8.4,8.85){$h$}
\put(5.85,11.4){$1$}
\put(5.2,-.75){{\rm Fig.~6}}
\end{picture}
\end{center}

\vspace*{4mm}

Clearly, $\mathbf P$ is not modular. Since
\begin{align*}
L\big(U(a,y),b\big) & =L(d,h,b)=L(e), \\
U\big(a,L(y,b)\big) & =U(a,c,g)=U(f), \\
L\big(U(a,y),x\big) & =L(d,h,x)=L(a), \\
U\big(x,L(y,b)\big) & =U(x,c,g)=U(b),
\end{align*}
the assumptions of Theorem~\ref{th5} as well as {\rm(ii)} of this theorem are satisfied and hence $e$ and $f$ are relative complements of $x$ in $[a,b]$.
\end{example}

Similarly as for lattices, if $y$ is a complement of $x$ then the existence of relative complements of $x$ in $[a,b]$ is assured by certain modular triples as follows.

\begin{theorem}\label{th8}
Let $\mathbf P=(P,\leq)$ be a poset, $a,b\in P$ with $a\leq b$, $x\in[a,b]$ and $y$ a complement of $x$ and assume that $e$ is the greatest element of $L\big(U(a,y),b\big)$ and $f$ the smallest element of $U\big(a,L(y,b)\big)$. Then $a\leq f\leq e\leq b$ and the following are equivalent:
\begin{enumerate}[{\rm(i)}]
\item $e,f\in R(a,b,x)$,
\item $(a,y,x)$ and $(x,y,b)$ are modular triples.
\end{enumerate}
\end{theorem}

\begin{proof}
Since $U(a,y)\cup\{b\}\subseteq U\big(a,L(y,b)\big)=U(f)$ we have
\[
f\in L(f)=LU(f)\subseteq L\big(U(a,y),b\big)=L(e)
\]
and hence $a\leq f\leq e\leq b$. Because of Theorem~\ref{th5}, (i) is equivalent to
\begin{equation}\label{equ5}
L\big(U(a,y),x\big)=L(a)\text{ and }U(b)=U\big(x,L(y,b)\big).
\end{equation}
Clearly, (\ref{equ5}) is equivalent to
\begin{equation}\label{equ6}
L\big(U(a,y),x\big)=L(a)\text{ and }L(b)=LU\big(x,L(y,b)\big).
\end{equation}
Since
\begin{align*}
L(a) & =LU(a)=L\big(U(a)\cap P\big)=L\big(U(a)\cap UL(y,x)\big)=LU\big(a,L(y,x)\big), \\
L(b) & =P\cap L(b)=LU(x,y)\cap L(b)=L\big(U(x,y),b\big), \\
\end{align*}
(\ref{equ6}) is equivalent to
\begin{equation}\label{equ7}
L\big(U(a,y),x\big)=LU\big(a,L(y,x)\big)\text{ and }L\big(U(x,y),b\big)=LU\big(x,L(y,b)\big).
\end{equation}
Finally, due to the definition of modular triples, (\ref{equ7}) is equivalent to (ii).
\end{proof}

If $y$ is not assumed to be a complement of $x$, but the triple $(a,y,b)$ is modular then we can determine a relative complement of $x$ in $[a,b]$ and, as in Theorem~\ref{th3} for lattices, again $e=f$.

\begin{theorem}\label{th7}
Let $\mathbf P=(P,\leq)$ be a poset, $a,b\in L$ with $a\leq b$, $x\in[a,b]$ and $y\in P$ and assume $(a,y,b)$ to be a modular triple and $e$ to be the greatest element of $L\big(U(a,y),b\big)=LU\big(a,L(y,b)\big)$. Then the following are equivalent:
\begin{enumerate}[{\rm(i)}]
\item $e\in R(a,b,x)$,
\item $L\Big(U\big(a,L(y,b)\big),x\Big)=L(a)$ and $U\Big(x,L\big(U(a,y),b\big)\Big)=U(b)$.
\end{enumerate}
\end{theorem}

\begin{proof}
Since
\begin{align*}
L(e) & =L\big(U(a,y),b\big)=LU\big(a,L(y,b)\big), \\
U(e) & =UL\big(U(a,y),b\big)
\end{align*}
and
\begin{align*}
L(e,x) & =L(e)\cap L(x)=LU\big(a,L(y,b)\big)\cap L(x)=L\Big(U\big(a,L(y,b)\big),x\Big), \\
U(x,e) & =U(x)\cap U(e)=U(x)\cap UL\big(U(a,y),b\big)=U\Big(x,L\big(U(a,y),b\big)\Big),
\end{align*}
(i) and (ii) are both equivalent to $L(e,x)=L(a)$ and $U(x,e)=U(b)$.
\end{proof}

If we add the assumption that $y$ is a complement of $x$, then we obtain the following result.

\begin{proposition}
{\rm(\cite{CM})} Let $\mathbf P=(P,\leq)$ be a modular poset, $a,b\in P$ with $a\leq b$, $x\in[a,b]$ and $y$ a complement of $x$ and assume $e$ to be the greatest element of $L\big(U(a,y),b\big)=LU\big(a,L(y,b)\big)$. Then $e\in R(a,b,x)$.
\end{proposition}

The following theorem generalizes this proposition and Theorem~\ref{th3}.

\begin{theorem}
Let $\mathbf P=(P,\leq)$ be a poset, $a,b\in L$ with $a\leq b$, $x\in[a,b]$ and $y$ a complement of $x$ and assume $(a,y,b)$ to be a modular triple and $e$ to be the greatest element of $L\big(U(a,y),b\big)=LU\big(a,L(y,b)\big)$. Then the following are equivalent:
\begin{enumerate}[{\rm(i)}]
\item $e\in R(a,b,x)$,
\item $L\Big(U\big(a,L(y,b)\big),x\Big)=LU\big(a,L(a,b,x)\big)$ and $L\big(U(x,a,y),b\big)=LU\Big(x,L\big(U(a,y),b\big)\Big)$.
\end{enumerate}
\end{theorem}

\begin{proof}
Because of Theorem~\ref{th7}, (i) is equivalent to
\begin{equation}\label{equ8}
L\Big(U\big(a,L(y,b)\big),x\Big)=L(a)\text{ and }U(b)=U\Big(x,L\big(U(a,y),b\big)\Big).
\end{equation}
Clearly, (\ref{equ8}) is equivalent to
\begin{equation}\label{equ9}
L\Big(U\big(a,L(y,b)\big),x\Big)=L(a)\text{ and }L(b)=LU\Big(x,L\big(U(a,y),b\big)\Big).
\end{equation}
Finally, since
\begin{align*}
L(a) & =LU(a)=L\big(U(a)\cap P\big)=L\big(U(a)\cap UL(y,x)\big)=LU\big(a,L(y,x)\big)= \\
     & =LU\big(a,L(y,b,x)\big), \\
L(b) & =P\cap L(b)=LU(x,y)\cap L(b)=L\big(U(x,y),b\big)=L\big(U(x,a,y),b\big),
\end{align*}
(\ref{equ9}) is equivalent to (ii).
\end{proof}

Authors' addresses:

Ivan Chajda \\
Palack\'y University Olomouc \\
Faculty of Science \\
Department of Algebra and Geometry \\
17.\ listopadu 12 \\
771 46 Olomouc \\
Czech Republic \\
ivan.chajda@upol.cz

Helmut L\"anger \\
TU Wien \\
Faculty of Mathematics and Geoinformation \\
Institute of Discrete Mathematics and Geometry \\
Wiedner Hauptstra\ss e 8-10 \\
1040 Vienna \\
Austria, and \\
Palack\'y University Olomouc \\
Faculty of Science \\
Department of Algebra and Geometry \\
17.\ listopadu 12 \\
771 46 Olomouc \\
Czech Republic \\
helmut.laenger@tuwien.ac.at
\end{document}